\documentclass[a4paper,10pt,leqno]{article}
\usepackage{amsmath,amsthm,amssymb}

\newtheorem{prop}{Proposition}[section]

\newtheorem{thm}[prop]{Theorem}
\newtheorem{cor}[prop]{Corollaly}
\newtheorem{lem}[prop]{Lemma}
\newtheorem{exam}[]{Example}
\newtheorem{remark}[prop]{Remark}

\numberwithin{equation}{section}

\newcommand{\R}{\mathbf{R}}
\newcommand{\C}{\mathbf{C}}
\newcommand{\N}{\mathbf{N}}
\newcommand{\Z}{\mathbf{Z}}
\newcommand{\Q}{\mathbf{Q}}

\newcommand{\8}{\infty}
\newcommand{\pl}{\partial}

\newcommand{\ds}{\displaystyle}

\newcommand{\ra}{\rangle}
\newcommand{\la}{\langle}

\begin{document}

\title{A remark on the Schr\"odinger equation on Zoll manifolds}
\author{Hisashi Nishiyama}
\date{}
\maketitle

\begin{abstract}
On the one dimensional sphere, the support of the fundamental solution to the Schr$\rm \ddot o$dinger equation consists of finite points at the time $t\in 2\pi\Q$.
 The paper \cite{Ka} generalized this fact to compact symmetric spaces. In this paper, we consider similar results on Zoll manifolds. We study the singularity for a solution to the equation using a functional calculus of the self-adjoint operator with integer eigenvalues.
\end{abstract}
\footnotetext[1]{2000 Mathematical Subject Classification. 35B65, 35P05, 35Q40, 58J40, 58J47.}
\section{Introduction}
We consider the Schr$\rm \ddot o$dinger equation on Zoll manifolds. First we recall the setting of the problem. We study the following,
\begin{equation}
(i\pl_t+\Delta_M)u=0.
\end{equation}
Here $i=\sqrt{-1}$ is the imaginary unit. $t$ is the time variable and $\pl_t=\frac{\pl}{\pl t}$. $M$ is an $d$-dimension compact smooth Riemannian manifold without boundary. $\Delta_M$ denotes the Laplace-Beltrami operator, i.e. in a local coordinate $x=(x_1,\dots,x_d)\in M$, 
\begin{equation*}
\Delta_M=\sum_{i,j=1}^d\frac{1}{\sqrt{{\rm det}(g_{ij})}}\pl_{x_i} g^{ij}(x)\sqrt {{\rm det}(g_{ij})}\pl_{x_j}.
\end{equation*}
Here the matrix $(g^{ij}(x))>0$ is the inverse of the metric $(g_{ij}(x))$ and ${\rm det}(g_{ij})$ is the determinant of $(g_{ij}(x))$. 
Since $M$ is compact, $\Delta_M$ is a self-adjoint operator in $L^2(M)$, the space of the square integrable functions on $M$. By the spectral theory, we can write the solution operator to the equation (1.1) as $e^{it\Delta_M}$. 
We also assume $M$ is the Zoll manifold, i.e. all geodesics has same minimal period. 
In a local coordinate, the geodesic flow is generated by the Hamiltonian vector field on $T^*M\backslash \{0\}$,
\begin{equation}
H=\pl_\xi l\cdot\pl_x-\pl_x l\cdot\pl_\xi
\end{equation}
were the function $l$ is defined by $l(x,\xi)=\sqrt{\sum_{i,j=1}^d  g^{ij}(x)\xi_i\xi_j}$ for $x\in M$, $\xi\not=0$. 
Without loss of generality, we can assume the geodesic flow has minimal period $2\pi$. 
Examples of Zoll manifolds are so called CROSSes (Compact Rank One Symmetric Spaces). 
For the background of Zoll manifolds, we refer \cite{Be}, \cite{Gu1}. 

The following example is our starting point. 
\begin{exam}({\rm \cite{Ka}, S. Doi}) 
Let $M=S^1\simeq \R/2\pi\Z$. We consider the following equation
\begin{equation}
\begin{cases}
(i\pl_t+\pl_x^2)G=0\\
G|_{t=0}=\delta (x).
\end{cases}
\end{equation}
Here $\delta(x)$ denotes the Dirac measure at $x=0$. 
Then the fundamental solution $G(t,x)$ satisfies the following two statement

(i) For fixed $t/2\pi\in\Q$,  ${\rm supp} (G(t,x))$ is finite points. 

(ii) For fixed $t/2\pi\in\R\backslash\Q$, ${\rm supp} (G(t,x))$ is the whole space $S^1$

\begin{proof}
First we prove the statement (i). By using the Fourier series, the solution can be written as follows
\begin{equation}
G(t,x)=\frac{1}{2\pi}\sum_{k\in\Z}e^{-itk^2+ikx}.
\end{equation}
Taking $t=2\pi n/m$, $n,m\in\Z$, and $k=mj+l$; $j\in\Z$, $0\leq l\leq m-1$. We have
\begin{align*}
G(2\pi n/m,x)&=\frac{1}{2\pi}\sum_{j\in\Z}\sum_{l=0}^{m-1}e^{-i2\pi\frac{n}{m}(mj+l)^2+i(mj+l)x}\\
&=\sum_{l=0}^{m-1}e^{-i2\pi\frac{n}{m}l^2+ilx}\frac{1}{2\pi}\sum_{j\in\Z}e^{imjx}\\
&=\sum_{l=0}^{m-1}e^{-i2\pi\frac{n}{m}l^2+ilx}\frac{1}{m}\sum_{j=0}^{m-1}\delta (x-2\pi j/m)
\end{align*}
For final identity, we use the Poisson summation formula for $2\pi$ periodic functions. We define
\begin{equation}
g(n,m;j)=\frac{1}{m}\sum_{l=0}^{m-1}e^{-i2\pi\frac{n}{m}l^2+il2\pi\frac{j}{m}}=\frac{1}{m}\sum_{l=0}^{m-1}e^{2\pi i(jl-nl^2)/m}.
\end{equation}
We have the following representation of the solution
\begin{equation}
G(2\pi n/m,x)= \sum_{j=0}^{m-1}g(n,m;j)\delta (x-2\pi j/m).
\end{equation}
This imply (i). We prove (ii). We use the following symmetry
\begin{lem}
$G(t,x)$ defined by (1.4) satisfies 
\begin{align} 
&G(t,x+2t)=e^{i(x+t)}G(t,x)\\
&G(t,x)=G(t,-x)
\end{align}
as the sense of distributions.
\end{lem}
The proof of this lemma is a direct computation. Recall
$$G=e^{it\pl_{x}^2}\delta\in H^{\frac{1}{2}(S^1)+\epsilon}\backslash H^{\frac{1}{2}}(S^1)$$
for any $\epsilon>0$. Here $H^s(S^1)$, $s\in\R$ are usual Sobolev spaces.
Since $S^1$ is compact, there exists a point $x_0$ such that near this point $G(t,x)$ does not have $H^{\frac{1}{2}}$ regularity. So by $(1.7)$, $G(t,x)$ does not belong to $H^{\frac{1}{2}}$ near $x_0+2t$. Inductively we can prove similar singularity at $x_0+2kt$, $k\in\Z$. If $t/2\pi=\alpha\in\R\backslash\Q$,
 these are dense in $S^1$. Thus
$${\rm sing\ supp} (G(2\pi\alpha,x))=S^1.$$
\end{proof}
\end{exam}
\begin{remark}
In \cite{Ka}, Kakehi calculates $g(n,m;j)$. He shows that for $n,m\in\Z$, $m>0$, $n,m$: co-prime, 
\begin{equation}
\begin{split}
g(n,m;j)\not=0, \ \text{$\forall$j   if $\ m\equiv 1, \ 3$  mod  $4$,}\\
\begin{cases}
&g(n,m;j)=0, \ \text{for j:even}\\
&g(n,m;j)\not=0, \ \text{for  j:odd,} 
\end{cases} \text{  if  $m\equiv 2$ mod  $4$,}\\
\begin{cases}
&g(n,m;j)\not=0,\  \text{for j:even}\\
&g(n,m;j)=0, \ \text{for j:odd}\\
\end{cases}\text{   if  $m\equiv 0$ mod  $4$}.
\end{split}
\end{equation}
Moreover he generalized this example to compact symmetric spaces. 
\end{remark}
In this paper, we prove an analogous of this example. 
Before to say our result, we give a physical interpretation of this finiteness of the support. 
In quantum mechanics, the equation (1.3) represent the motion of the free particle. 
On $S^1$, the plain wave is given by 
\begin{equation*}
\frac{1}{2\pi}e^{-itk^2+ikx}=\frac{1}{2\pi}e^{-ik(tk-x)},\  k\in\Z. 
\end{equation*}
This is the motion whose velocity is $k$. This means in $S^1$ the velocity is quantized. 
So at the time $t$, the particle can reach the point 
\begin{equation*}
k-x\equiv 0\ \mod \ 2\pi\ , k\in\Z.
\end{equation*}
Taking $t=2\pi p/q$, this can be written as follows
\begin{equation*}
x\equiv 2\pi \frac{p}{q}k\ \mod \ 2\pi\ ,k\in\Z.
\end{equation*}
These are finite points. By this naive interpretation, it is natural to expect same phenomenon for Zoll manifolds since $\sqrt{-\Delta_M}$ has asymptotically integer eigenvalues and whose geodesics are periodic. The main result of this article is the following. 
\begin{thm}
Let M be a Zoll manifolds with period $2\pi$ and $x_0\in M$. We consider the following  equation
\begin{equation}
\begin{cases}
(i\pl_t+\Delta_M)u=0.\\
u|_{t=0}=\delta_{x_0}.
\end{cases}
\end{equation}
Here $\delta_{x_0}$ is the Dirac measure on $x_0$. The following two statement holds to this solution

(i) For fixed $t/2\in\Q$, ${\rm sing\ supp} (u(t, x))$ is degenerate i.e. subset of finite $d-1$ dimensional sub-manifolds.  

(ii) For fixed $t/2\pi\in\R\backslash\Q$, ${\rm sing\ supp}(u(t,x))$ is the whole space $M$
\end{thm}
\begin{remark}
Moreover we compute the wave front set of the solution. We can also compute the wave front set to the equation with a first order perturbation. 
For more precise statement, see Corollary 3.4 and Theorem 4.1.
\end{remark}
We should remark some related results. In $R^n$, the singularity propagates at infinite speed and the smoothing effect may occur. 
There are many results on the singularity of the solution to the Schr$\rm \ddot o$dinger equation related to the  smoothing effect. 
For example, if the classical trajectories are not trapped and the potential decays faster than the quadratic order, then 
the solution with compact support initial data is smooth at $t>0$ see e.g. \cite{Na}. 
On the other hand to compact manifolds, there are few works to the study of the singularity except the work \cite{Ka}. 
On $\R^1$, Yajima \cite{Ya} proved fundamental solution is singular on the whole space time, to the super quadratic potential. 
Our result indicates similar result on Zoll manifolds. 

\section{Preliminaries}
In this paper,  we use the microlocal analysis on manifolds. We recall some relevant notations, see \cite{Gr-Sj}, \cite{Ho}.
$S^m(T^*M)$, $m\in\R$ is the space of functions $a(x,\xi)$ on $T^*M$, which are smooth in $(x, \xi)$ and satisfy
$$\pl^\alpha _x \pl^\beta _\xi a(x,\xi)= {\cal O}((1+|\xi|)^{m-|\beta|})\ ,  (x,\xi) \in T^*M.$$
We also write $S^{-\8}(T^*M)=\cap_m S^m(T^*M).$
If $a\in S^m(T^*M)$, we can define correspondent pseudo-differential operators, for short $\Psi DO$, by usual way. 
$\Psi^m(M)$ denotes the class of $\Psi DO$ associated with this class. The smoothing operator is a operator which has smooth kernel. A operator in $\Psi^{-\8}(M)$ is the smoothing operator.
It is useful to consider the operator as the operator on the half-density bundle. 
For the operator $A\in \Psi^m(M)$ and Riemannian density $\rho$
$$Bu=\rho^{\frac{1}{2}}A(u\rho^{-\frac{1}{2}}),\ u\in C^\8(M, \Omega^\frac{1}{2})$$
defines $B$, the operator on the half-density bundle. 
We write this class of operator as $\Psi^m(M;\Omega^\frac{1}{2}, \Omega^\frac{1}{2})$. By this identity, we identify the operator on $M$ with the operator on the half-density bundle. 

For $A\in \Psi^m(M;\Omega^\frac{1}{2}, \Omega^\frac{1}{2})$, we have the symbol map
$\sigma(A)\in S^m(T^*M)/S^{m-2}(T^*M)$ and there exist $a_0 \in S^m(T^*M) $, $a_1 \in S^{m-1}(T^*M) $ such that $\sigma(A)-a_0-a_1\in S^{m-2}(T^*M) $. We call $a_0=\sigma_{pri}(A)$ the principal symbol of $A$ and $a_1=\sigma_{sub}(A)$ the sub-principal symbol of $A$. 
Modifying above definition, $\Psi_{phg}^m(M; \Omega^\frac{1}{2}, \Omega^\frac{1}{2})$ denotes the class of $\Psi DO$ whose principal and sub-principal symbols are polyhomogeneous functions of order $m$ and $m-1$ respectively. $\sqrt{-\Delta_M}$ is a $\Psi DO$ $\in \Psi^{1}_{phg}(M; \Omega^\frac{1}{2}, \Omega^\frac{1}{2})$ whose principal symbol is $(\sum_{i,j=1}^d  g^{ij}(x)\xi_i\xi_j)^{1/2}$ and sub-principal symbol is $0$. 

We also use the notation of the wave front set. If $A\in\Psi_{phg}^0(M)$ has the principal symbol $a$, the characteristic set of $A$ is given by 
$${\rm Char} A=\{(x,\xi)\in T^*M\backslash \{0\}, a=0\}$$
We define the wave front set of $u \in {\cal D'}(M)$ by
$${\rm WF}(u)=\cap \{{\rm Char } A; A\in \Psi_{phg}^0(M), Au\in C^\8(M)\}$$
where  ${\cal D'}(M)$ is the space of distributions on $M$. By definition, the wave front set is a conic subset of $T^*M\backslash \{0\}$. 
The  wave front set is a refinement of the singular support, if $\Pi_x(x,\xi)$ is the projection to $x$ variable then
$$\Pi_x( {\rm WF}(u))={\rm sing \ supp} (u).$$
We use the following Propagation of singularity to the first order operator. 
\begin{prop}
If $A\in \Psi^1(M;\Omega^\frac{1}{2}, \Omega^\frac{1}{2})$ has real-principal symbol. We have
\begin{equation}
{\rm WF}( e^{itA}u)={\exp}(tH_a){\rm WF}(u),   \text{$  \forall u \in {\cal D'}(M)$}.
\end{equation}
Here ${\rm exp}(t H_a )$is the Hamilton flows for the symbol $a=\sigma_{pri}(A)$.
\end{prop}
We use a functional calculus to the self-adjoint operator with integer eigenvalues. By the following proposition, we can reduce 
$\sqrt{-\Delta_M}$ to this type operator. This result is due to Colin de Verdi$\rm\grave e$re, Duistermaat, Guillemin, and Weinstein-see \cite{Ho} and \cite{Gu2} for expositions of the spectral theory of the Laplace-Bertrami Operator on Zoll manifolds. 

\begin{prop}
Let M be a Zoll manifolds with period $2\pi$. 
Then there exists a integer $\alpha$ and a self-adjoint operator $E \in \Psi^{-1}_{phg}(M; \Omega^\frac{1}{2}, \Omega^\frac{1}{2})$ which commute with $\Delta_M$ such that
\begin{equation}
{\rm Spec}(\sqrt{-\Delta_M}+\alpha/4 +E)\subset \Z, 
\end{equation}
\end{prop}
For the proof, see Lemma 29.2.1 in \cite{Ho}. $\alpha$ is so called the Maslov index, in our case this is the number of conjugate points in a closed geodesic counted with multiplicity.    
By this proposition, we should consider the self-adjoint operator with integer eigenvalues. 
The following special case of the functional calculus is essential to our results. 
\begin{prop}
Let $L$ be a self-adjoint operator on the Hilbert space H and $f$ be a function $\Z\rightarrow \C$. We assume ${\rm Spec}(L)\subset \Z $. Then we have the identity
\begin{equation}
f(L)=\sum_{k\in\Z}\frac{1}{2\pi}\int^{2\pi}_0f(k)e^{iky}e^{-iyL}dy
\end{equation}
\end{prop}
\begin{proof}
This is proved by  
\begin{equation}
\frac{1}{2\pi}\int^{2\pi}_0e^{i(k-l)y}dy=\delta_{k,l}
\end{equation}
and the following identity
\begin{align}
f(L)=\sum_{k\in\Z}f(k)P_k, e^{-iyL}=\sum_{k\in\Z}e^{-iyk}P_k.
\end{align}
Here $\delta_{k,l}=1$ if $k=l$ and otherwise $\delta_{k,l}=0$, $P_k$, $k\in\Z$ are orthogonal projections to the space of eigenfunctions with eigenvalue $k$ i.e. 
\begin{equation}
\begin{split}
\begin{cases}
P_k: \text{projections to eigenfunctions with eigenvalue $k$},\\
 \ \ \ \  \text{if $k$ is an eigenvalue of $L$}\\
P_k: \ 0, \ \text{if $k$ is not an eigenvalue of $L$}.
\end{cases}
\end{split}
\end{equation}
\end{proof}
In the above identity, we shall change the sum and integral sign. 
We use self-adjoint operators $\la D_x\ra^{s}$ on $L^2(S^1)$ defined
 by the Fourier series, 
\begin{equation}
\la D_x\ra^{s}f(x)=\sum_{k\in\Z}\la k\ra^sf_ke^{ikx}.
\end{equation}
Here $\la k\ra=(k^2+1)^{1/2}$, $s\in\R$ and $f(x)=\sum_{k\in\Z} f_ke^{ikx}$. 
This definition also can apply to the distribution on the $S^1$. We also define linear operators on $H$
\begin{equation}
\la L\ra^{s}=\sum_{k\in\Z}\la k\ra^sP_k.
\end{equation}
Since $P_k$ are orthogonal projections, we get
\begin{equation}
\begin{split}
\la D_y\ra^{N}& e^{-iyL}=\sum_{l\in\Z} \la l\ra^{N}P_le^{-ily}\\
&=\sum_{l\in\Z}\la l\ra^{N}P_l  \sum_{j\in\Z}P_je^{-ijy}= \la L\ra^{N} e^{-iyL}.
\end{split}
\end{equation}
\begin{prop}
Let $L$ be a self-adjoint operator on the Hilbert space H and $f$ be a function $\Z\rightarrow \C$. We assume ${\rm Spec}(L)\subset \Z $. Then we have the identity
\begin{equation}
f(L)=\la L\ra^{N}\frac{1}{2\pi}\int^{2\pi}_0\sum_{k\in\Z}\la k\ra^{-N}f(k)e^{iky} e^{-iyL}dy
\end{equation}
if $\sum_{k\in\Z}\la k\ra^{-N}f(k)e^{iky}$ is absolute uniform convergent for some $N\in\R$.
\end{prop}
\begin{proof}
By Proposition 2.2 and (2.3), we have
\begin{equation}
\begin{split}
f(L)&=\sum_{k\in\Z}\frac{1}{2\pi}\int^{2\pi}_0 \la k\ra^{-N}f(k)e^{iky} \la k\ra^{N}\sum_{l\in\Z}P_le^{-ily}dy\\
&=\sum_{k\in\Z}\frac{1}{2\pi}\int^{2\pi}_0 \la k\ra^{-N}f(k)e^{iky} \sum_{l\in\Z} \la l\ra^{N}P_le^{-ily}dy\\
\end{split}
\end{equation}
Here $P_k$ are defined by (2.5). From (2.8), we get
\begin{align*}
f(L)=\la L\ra^{N}\sum_{k\in\Z}\frac{1}{2\pi}\int^{2\pi}_0\la k\ra^{-N}f(k)e^{-iyL}dy.
\end{align*}
Since $\sum_{k\in\Z}\la k\ra^{-N}f(k)e^{iky}$ is absolute uniform convergent, we can change the sum and integral
\begin{align*}
f(L)=\la L\ra^{N}\frac{1}{2\pi}\int^{2\pi}_0\sum_{k\in\Z}\la k\ra^{-N}f(k)e^{iky} e^{-iyL}dy.
\end{align*}
\end{proof}
Especially we have the identity
\begin{equation}
\begin{split}
e^{-itL^2}&=\la L\ra^{N}\frac{1}{2\pi}\int^{2\pi}_0\la D_y\ra^{-N}G(t,y)e^{-iyL}dy
\end{split}
\end{equation}
for $G(t,y)$ defined by (1.4) and $N>1$. By (1.8), this can be written as follows
\begin{equation}
\begin{split}
e^{-itL^2}&=\la L\ra^{N}\frac{1}{2\pi}\int^{\pi}_0 \la D_y\ra^{-N}G(t,y)(e^{iyL}+e^{-iyL})dy\\
&=\la L\ra^{N}\frac{1}{2\pi}\int^{2\pi}_0\la D_y\ra^{-N}G(t,y){\rm cos}(iyL)  dy
\end{split}
\end{equation}
The following proposition is an abstract form of Theorem 1.3 (i). 
\begin{prop}
Let $L$ be a self-adjoint operator on the Hilbert space H. We assume ${\rm Spec}(L)\subset \Z $. 
Then for the time $t= 2\pi n/m \in 2\pi\Q$, $n,m\in\Z$, we have the following identity, 
\begin{equation}
U(2\pi n/m)=\sum_{j=0}^{m-1}g(n,m;j)V(2\pi j/m)
\end{equation}
Here $U(t)=e^{-itL^2}$, $V(t)=e^{-itL}$ and $g(n,m,j)$ is defined by (1.5).
\end{prop}
\begin{proof}
By using the identity (2.12) and (1.6), we have 
\begin{align*}
U(2\pi n/m)&=\frac{1}{2\pi}\int^{2\pi}_0\la D_y\ra^{-N}\sum_{j=0}^{m-1}g(n,m;j)\delta (x-2\pi j/m) \la L\ra^{N}e^{-iyL}dy\\
&=\frac{1}{2\pi}\int^{2\pi}_0\la D_y\ra^{-N}\sum_{j=0}^{m-1}g(n,m;j)\delta (x-2\pi j/m) \la D_y\ra^{N}e^{-iyL}dy\\
&=\frac{1}{2\pi}\int^{2\pi}_0\sum_{j=0}^{m-1}g(n,m;j)\delta (x-2\pi j/m) e^{-iyL}dy\\
&=\sum_{j=0}^{m-1}g(n,m;j)V(2\pi j/m).
\end{align*}
\end{proof}
\begin{remark}
By (2.13), we can prove the similar result for $V(t)={\rm cos}(itL)$.
\end{remark}
\begin{remark}
We can prove this proposition elementary from the following lemma which is proved by elementary linear algebra. 
\begin{lem}\label{}
Let $L$ and $V(t)$ be as above and $\mathbf{P}_l=\sum_{j\in\Z} P_{mj+l}.$. There exists an invertible
 $m \times m$ complex matrix $(a_{ij})$, $a_{ij}\in\C$ such that the following identity holds
\begin{equation}
\mathbf{P}_i=\sum_{j=0}^{m-1}a_{ij}V(2\pi j/m).
\end{equation}
\end{lem}
\end{remark}
Later we use the wave front set of $G(t,x)$
\begin{prop}
If $t/2\pi=\alpha\in R\backslash\Q$, we have
\begin{equation}
{\rm WF} (G(2\pi\alpha,x))=T^*S^1\backslash\{0\}.
\end{equation}
\end{prop}
for $G(t,x)$ defined by (1.4).
\begin{proof}
By example 1(ii), the singularity exists at $x=0$. So we have $(0, \xi_0) \in{\rm WF} (G(2\pi\alpha,y))$ for some $\xi_0\not=0$. By (1.8), We also get $(0, -\xi_0) \in{\rm WF} (G(2\pi\alpha,y))$. 
We apply similar argument used to prove Example 1 (ii). 
\end{proof}
\section{Proof of Theorem 1.3 (i).}
In this section, we consider $t/2\pi\in\Q$ case. We prove the following theorem.
\begin{thm}
Let M be a Zoll manifold with period $2\pi$, $P_s\in\Psi_{phg}^1(M; \Omega^\frac{1}{2}, \Omega^\frac{1}{2})$ be a self-adjoint operator and $P=\Delta_M+P_s$.
 Then there exit self-adjoint operators $Q_1$, $Q_2\in\Psi_{phg}^1(M; \Omega^\frac{1}{2}, \Omega^\frac{1}{2}) $ which is commute with $\Delta_M$ and $S\in\Psi^0(M; \Omega^\frac{1}{2}, \Omega^\frac{1}{2})$ such that 
at the time $t= 2\pi n/m \in 2\pi\Q$,  $n,m\in\Z$, 
the following identity holds
\begin{equation}
U(2\pi n/m)=e^{-iS}W(2\pi n/m)\sum_{j=0}^{m-1}g(n,m;j)V(2\pi j/m){e^{iS}}+Smoothing.
\end{equation}
where $U(t)=e^{itP}$, $V(t)=e^{-itQ_1}$ and $W(t)=e^{itQ_2}$. 
The principal symbols of these operator are given by
 $\sigma_{pri} (Q_1)$=$\sigma_{pri} (\sqrt{-\Delta_M})$, $\sigma_{pri} (Q_2)=\sigma_{pri}(\tilde {P_s}+\alpha\sqrt{-\Delta_M} /2) $. 
Here
\begin{equation}
\sigma_{pri}(\tilde {P_s})=\frac{1}{2\pi}\int^{2\pi}_{0}\sigma_{pri}(P_s)({\rm exp}(tH_l))dt
\end{equation}
where ${\rm exp}(tH_l)$ is the Hamilton flow for the symbol $l=\sigma_{pri} (\sqrt{-\Delta_M})$. 
$S$ satisfies 
$$e^{iS}(\Delta_M+P_s{)e^{-iS}}=\Delta_M+\tilde{P_s}+R$$
for a self-adjoint operators $\tilde{P_s}\in\Psi_{phg}^1(M; \Omega^\frac{1}{2}, \Omega^\frac{1}{2})$ which is commute with $\Delta_M$ and $R\in \Psi^{-\8}(M; \Omega^\frac{1}{2}, \Omega^\frac{1}{2})$. $\alpha $ denotes the Maslov index.
The coefficient $g(n,m;j)\in\C$ is defined by $(1.5)$.
\end{thm}
\begin{remark}
 $e^{iS}$ does not change the wave front set since it is invertible and it is a $\Psi DO$ whose principal symbol is defined by $e^{i\sigma_{pri}(S)}\in S^0(M)$. 
For $V(t)$ and $W(t)$, we can apply the propagation of singularity. 
\end{remark}
Thus we have the propagation of the wave front set.
\begin{cor}
Let M and $P=\Delta_M+P_s$ as above. 
Then at the time $t=2\pi n/m$, $n,m\in\Z$, $m>0$, $n,m$: co-prime, we have
\begin{equation}
{\rm{WF}}(U(2\pi n/m)u)\subset {\rm exp}(2\pi n/m H_q)\sum_{j; g(n,m;j)\not=0} {\rm exp}(2\pi j/m H_l ){\rm{WF}}(u)  \text{,  $\forall u \in {\cal D'}(M)$}.
\end{equation}
Here ${\rm exp}(t H_l )$ and ${\rm exp}(t H_q)$ are the Hamilton flows for the symbols $l=\sigma (\sqrt{-\Delta_M})$, and $q=\sigma (\alpha \sqrt{-\Delta_M}/2+\tilde {P_s})$ respectively.
\end{cor}
The term $j; g(n,m;j)\not=0$ can be computed by (1.9). Similarly we get
\begin{cor}
Let M and $P=\Delta_M+P_s$ as above. 
Then for the time $t=2\pi n/m$, $n,m\in\Z$, $m>0$, $n,m$: co-prime, the following identity holds
\begin{equation}
{\rm{WF}}(U(2\pi n/m)\delta_{x_0})={\rm exp}(2\pi n/m H_q)\sum_{j; g(n,m;j)\not=0}{\rm exp}(2\pi j/m H_l ) {\rm{WF}}(\delta_{x_0}).
\end{equation}
Here $\delta_{x_0}$ is the Dirac measure at $x_0\in M$ and ${\rm exp}(t H_l )$ and ${\rm exp}(t H_q)$ are as above.
\end{cor}
This is a generalization of Theorem 1.3 (i). We prove Theorem 3.1 by reducing $e^{itP}$ to the form to which we can apply Proposition 2.5. This is essentially well known but for readers we write the detail. First we will assume $P_s=0$. By Proposition 2.2, we take a self adjoint operator  $E \in \Psi^{-1}_{phg}(M; \Omega^\frac{1}{2}, \Omega^\frac{1}{2})$ such that 
\begin{equation}
{\rm Spec}(\sqrt{-\Delta_M}+\alpha/4 +E)\subset \Z
\end{equation}  
We write 
\begin{equation}
L=\sqrt{-\Delta_M}+\alpha/4 +E.
\end{equation}
Then we have
$$L^2= -\Delta_M + \alpha\sqrt{-\Delta_M}/2+\tilde E$$
Here $\tilde E\in\Psi^{0}_{phg}(M; \Omega^\frac{1}{2}, \Omega^\frac{1}{2})$ commutes with $\Delta_M$. Since $\tilde E$ commute with $\Delta_M$, we have
\begin{equation}
e^{it\Delta_M }=e^{it (\alpha\sqrt{-\Delta_M}/2+\tilde E)}e^{-i tL^2}\\
\end{equation}
Next we shall consider the case $P_s\not =0$. First we assume that $P_s$ commutes with $\Delta_M$. Then by (3.7), we have
\begin{equation}
e^{it\Delta_M+P_s }=e^{it (\alpha\sqrt{-\Delta_M}/2+P_s+\tilde E)}e^{-i tL^2}\\
\end{equation}
Finally we consider the case that $P_s$ does not commute with $\Delta_M$. 

We claim that there exit self-adjoint $S \in \Psi^{0}(M; \Omega^\frac{1}{2}, \Omega^\frac{1}{2})$ and $R\in \Psi^{-\8}(M; \Omega^\frac{1}{2}, \Omega^\frac{1}{2})$ such that 
\begin{equation}
e^{iS}(\Delta_M+P_s)e^{-iS} + R =\Delta_M+\tilde P_s.
\end{equation}
Here $\tilde P_s \in \Psi^{1}_{phg}(M; \Omega^\frac{1}{2}, \Omega^\frac{1}{2})$ is a self-adjoint operator which commutes with $\Delta_M$. The principal symbol of $\tilde P_s$ is given by $(3.2)$. 

This is essentially due to \cite{We}, see also \cite{Ho}.
Taking $Q=P_s- \alpha\sqrt{-\Delta_M}/2-\tilde E$, we will prove $L^2+Q$ is unitary equivalent to $L^2+B$ modulo smoothing operator where $B \in \Psi^{1}$, commutes with $\Delta_M$, is given below.   
We introduce the operator
$$Q_t=e^{it{L}}Qe^{-it{L}}$$
which is a self-adjoint $\Psi DO$. Taking an average,
$$B_1=\frac{1}{2\pi}\int_0^{2\pi}Q_tdt,$$
it follows that
$$[L,B_1]=\frac{1}{2\pi i}( e^{i2\pi {L}}Qe^{-i2\pi {L}} -Q).$$
since 
$$\frac{d}{dt}Q_t=ie^{it{L}}[L,Q]e^{-it{L}}.$$
By (3.5), we have $[L,B_1]=0$. Applying the Egolov's theorem the sub-principal symbol of $L^2+B_1$ is given by (3.2). 
We set $B= B_1.$
By similar calculus, we have
\begin{equation}
B-Q=\frac{1}{2\pi}\int_0^{2\pi}dt\int_0^t\frac{d}{ds}Q_sds=[iT, L]
\end{equation}
where
$$T=\frac{1}{2\pi}\int_0^{2\pi}dt\int_0^tQ_sds.$$
Since $L$ is elliptic, we can take its parametrix $N$.
Then $S_1=1/2TN$ is a zero order self-adjoint $\Psi DO$. By the calculus of $\Psi DO$, we have
$$[S_1,L^2]=\frac{1}{2}([T,L^2]N+T[N,L^2])=[T,L]+R_0$$
where $R_0\in \Psi^{-1}$ is a self-adjoint operator. 
We set $S=S_1$ and consider $ie^{i{S}}(L^2+Q)e^{-i{S}}$.

If $A\in\Psi^m$, then $ie^{i{S}}Ae^{-i{S}}$ is a $\Psi DO$ whose symbol is the asymptotic sum of the following formal series
$$\sum_0^\8(({\rm ad}\  iS)^jA/ j!$$
where $({\rm ad} \ iS)A=[iS,A]$.
And by (3.9), We have 
$$ie^{i{S}}(L^2+Q)e^{-i{S}}=L^2+B+R_1 $$
where $R_1\in \Psi^{0}$ is a self-adjoint operator. Applying similar argument to $R_1$, we have $B_2\in \Psi^{0}$ and $S_2\in \Psi^{-1}$ satisfying 
 $$ie^{i{S}}(L^2+Q)e^{-i{S}}=P^2+B+R_2 $$
 where $B=B_1+B_2$, $S=S_1+S_2$ and $R_2\in \Psi^{-1}$ are self-adjoint operators. Continuing this argument and taking asymptotic sums, we have desired $B$ and $S$. This means (3.9). Thus we have the identity
\begin{equation}
i\pl_t+\Delta_M+\tilde P_s=e^{iS}(i\pl_t+\Delta_M+P_s)e^{-iS}  +R.
\end{equation}
We shall give the solution operator to this equation. We write the solution operator as follows $$e^{iS}(e^{it( \Delta_M+ P_s)}+ A(t) )e^{-iS}$$ and inserting this to the right hand side of (3.11), we have the following equation to the operator $A(t)$,
\begin{equation}
\begin{cases}
(i\pl_t+\Delta_M+P_s+e^{-iS}Re^{iS}  )A(t) =- e^{-iS}Re^{iS} e^{it( \Delta_M+ P_s)}\\
A(0)=0.
\end{cases}
\end{equation}
This can be solved by Duhamel's principle,  we have
\begin{equation}
A(t)=i\int_0^te^{i(t-s)(\Delta_M+P_s+e^{-iS}Re^{iS} )}e^{-iS}Re^{iS} e^{is( \Delta_M+ P_s)}ds.
\end{equation}
Since $e^{-iS}Re^{iS}$ is a smoothing operator, $A(t)$ is a smoothing operator. We have the identity 
\begin{equation}
e^{it( \Delta_M+ P_s)}=e^{-iS}e^{it( \Delta_M+\tilde P_s)}e^{iS}+ A(t).
\end{equation} 
Since $\tilde P_s$ is commute with $\Delta_M$, by the identity and (3.8), we have proved the following
\begin{prop}
Let M be a Zoll manifold with period $2\pi$ and $P=\Delta_M+P_s$ as above. Then there exit self-adjoint operators $S, \ \tilde E\in\Psi^0(M; \Omega^\frac{1}{2}, \Omega^\frac{1}{2})$ such that 
\begin{equation}
e^{it( \Delta_M+ P_s)}=e^{-iS}e^{it (\alpha\sqrt{-\Delta_M}/2+\tilde P_s+\tilde E)}e^{-i tL^2}e^{iS}+ {\rm Smoothing}
\end{equation}
Here $L$ is defined by (3.6)
\end{prop}
We apply Proposition 2.5 to (3.15). Taking 
$Q_1=L$ and $Q_2=\alpha\sqrt{-\Delta_M}/2+\tilde P_s+\tilde E$, we have Theorem 3.1.
\begin{remark}
If the manifolds are odd dimension rank 1 symmetric spaces, we can easily prove support theorem by our identity
although Kakehi proved support theorem for more general compact symmetric spaces.
\end{remark}
We outline the argument. Essentially we can assume the manifolds are odd dimensional spheres. The spectrum of the Laplace-Beltrami operator is the following, see e.g. \cite{Ta}, 
\begin{equation}
{\rm spec}\Delta_{S^d}=\{-k(k+(d-1)/2);k\in\N\}.
\end{equation}
We have
$$\Delta_{S^d}= \Delta_{S^d}-(d-1)^2/4+(d-1)^2/4=-\sqrt{-\Delta_{S^d}+(d-1)^2/4}^2+(d-1)^2/4$$
By (3.16), we know $${\rm spec}\sqrt{-\Delta_{S^d}+(d-1)^2/4}\subset \Z$$
So we can apply the Remark 2.5 to this operator. Using the Huygens' principle, see \cite{La-Ph}, \cite{Ta}, 
we have the following
\begin{equation}
{\rm supp}(e^{2\pi i n/m\Delta_{S^d}}\delta_{x_0})=\sum_{j; G(n,m;j)\not=0} d(2\pi j/m){\rm supp}(\delta_{x_0})  \text{,  $\forall u \in {\cal D'}(M)$}.
\end{equation}
Here $\delta_{x_0}$ is the Dirac measure at $x_0\in S^d$, $d(t)$ is defined by 
$$d(t)\{x\}=\{y\in S^d; {\rm dist}_{S^n}(y,x)=t\},\ t\in\R, x\in S^n,  $$
where ${\rm dist}_{S^d}$ is the distance of the Sphere.
\section{Proof of Theorem 1.3 (ii).}
We shall consider $t/2\pi\in \R\backslash \Q$ case. 
We prove the following a generalization of Theorem 1.3 (ii). 
\begin{thm}
Let M be a Zoll manifolds with period $2\pi$, $P_s\in\Psi_{phg}^1(M; \Omega^\frac{1}{2}, \Omega^\frac{1}{2})$ be a self-adjoint operator and $P=\Delta_M+P_s$. 
If $t/2\pi =\alpha\in \R\backslash \Q$, the following identity holds
\begin{equation}
{\rm{WF}}(U(2\pi\alpha)\delta_{x_0})= {\rm exp}(2\pi \alpha H_q)\cup_{0\leq \tau \leq 2\pi} {\rm exp}(\tau H_p) {\rm{WF}}(\delta_{x_0})
\end{equation}
Here $\delta_{x_0}$ is the Dirac measure at $x_0\in M$, $U(t)\delta_{x_0}=e^{itP}\delta_{x_0}$,
${\rm exp}(t H_l )$ and ${\rm exp}(t H_q)$ are the Hamilton flows for the symbols $l=\sigma_{pri} (\sqrt{-\Delta_M})$, and $q=\sigma_{pri} (\tilde {P_s})$ respectively. $\sigma_{pri}(\tilde {P_s})$ is defined by (3.2).
\end{thm}
\begin{remark}
We can also prove similar identity to ${\rm WF }_{d/2}$. 
We notice the following
\begin{equation}
{\rm{WF}}(U(t)u)\subset {\rm exp}(t H_q)\cup_{0\leq \tau \leq 2\pi} {\rm exp}(\tau H_p) {\rm{WF}}(u)
\end{equation}
by (2.12) and Proposition 3.5. 
\end{remark}
\begin{remark}
 Since $M$ is complete, if $\sigma_{pri} (\tilde {P_s})=0$, then we have 
$${\rm{sing\ supp}}(U(2\pi\alpha)\delta_{x_0})=M$$ 
from Theorem 4.1. This implies Theorem 1.3 (ii).
\end{remark}

By Proposition 3.5, we study the singularity of $e^{-itL^2}\delta_{x_0}$ here $L$ is defined by (3.6). 
We fix the coordinates $(x, \xi)$ of the cotangent bundle where $x$ is the normal coordinate around $x_0=0$ and $\xi$ is the corespondent coordinate. 
We study $e^{-itL^2}\delta_{x_0}$ near $x_0$. We recall that $e^{-itL} $ can be approximated, modulo smooth functions, by the Fourier integral operator, see \cite{Gr-Sj}, \cite{Ho}
\begin{equation}
e^{itL}u\equiv (2\pi)^{-d/2}\int e^{i(\Phi(t, x,\xi)-y\cdot\xi)}a(t, x,\xi)u(y)dyd\xi\ \ {\rm mod}\ C^\8(M).
\end{equation}
From now on, we omit the term ${\rm mod}\ C^\8(M)$. Here the phase $\Phi(t, x,\xi)$ and the amplitude $a(t, x,\xi)$ are chosen to satisfy
\begin{align}
\begin{cases}
&(i\pl_t+L)e^{i\Phi(t, x,\xi)}a(t, x,\xi)= {\cal O}((1+|\xi|)^{-\8})\\
&e^{i\Phi(0, x,\xi)}a(0, x,\xi)=(2\pi)^{-d/2}e^{ix\cdot\xi}.
\end{cases}
\end{align}
These functions can be taken if $\Phi(t, x,\xi)$ satisfies the following Eikonal equation
\begin{equation}
\begin{cases}
&\pl_t\Phi+l(x,\pl_x\Phi)=0\\
&\Phi(0, x,\xi)=x\cdot\xi
\end{cases}
\end{equation}
here the symbol $l(x,\xi)$ is the principal symbol to the operator $L$ and we also inductively choose the amplitude $a(t, x,\xi)=a_0+a_1+\cdots$ by solving transport equations. 
For small $t$, these equations have solutions and we get $\Phi(t, x,\xi)$, a real $C^\8$ function homogeneous of degree one in $\xi$ and $a(t, x,\xi)=a_0+a_1+\cdots\in S_{phg}^{0}(M)$ where $a_j\in S^{-j}(M)$ is homogeneous of degree $-j$ in $\xi$. 
In (4.3) we take $u=\delta_{x_0}$, and get
$$e^{itL}\delta_{x_0}\equiv \int e^{i\Phi(t, x,\xi)}a(t, x,\xi)d\xi=:A(t,x).$$
Here the integral is defined by the oscillatory integral. Taking a point $x_1\not= x_0$ near $x_0$, we study $A(t,x)$ in some neighborhood of $x_1$.  
Change coordinates to polar coordinates $(r,\theta)$, $(\lambda,\omega)$ to $x$, $\xi$ respectively. We have
$$A(t,r\theta)=\int_0^\8\int_{S^(d-1)} e^{i\Phi (t,r\theta,\lambda\omega)}a(t,r\theta,\lambda\omega)\lambda^{d-1}d\lambda d\omega$$
We write $x_1=r_1\theta_1$. Here $r_1>0$ is sufficient small. We shall apply the stationary phase method to the $\omega$ integral. 
Since $\Phi$ satisfies (4.5), we have
$$\Phi(t, x,\xi)=x\xi-t(\sum_{i,j}g^{ij}\xi_i\xi_j)^{1/2}+ {\cal O}(t^2\xi).$$
On normal coordinate, the metric satisfies
\begin{equation}
g^{ij}(x)=\delta_{ij}+{\cal O}(|x|).
\end{equation}
We take $|x|=r$ the same order as $t$. Then we get
\begin{equation}
\Phi(t, x,\xi)=x\cdot\xi-t|\xi|+ {\cal O}(t^2\xi)=\lambda(r\theta \cdot\omega-t+{\cal O}(t^2))
\end{equation}
Since $\theta \cdot\omega$ take its maximum or minimum at $\omega=\pm\theta$ respectively,
 $\omega=\pm\theta$ is critical values of $\theta\cdot \omega$. Moreover they are only critical points and non-degenerate i.e. the Hessian is not degenerate. 
So we can assume the critical points of $\Phi(t, r\theta,\lambda\omega)$ is non-degenerate. By the stationary phase method, we obtain
\begin{align*}
A(t,r\theta)\equiv\sum_{j}\int_0^\8e^{i\lambda {\tilde\Phi }_j(t,r\theta )} {\tilde a}_j(t,r\theta,\lambda)\lambda^{(d-1)/2}d\lambda
\end{align*}
The phase function is defined by
$${\tilde\Phi }_j(t,r\theta )=\Phi(t, r\theta,\omega_j)$$
where $\omega_j$ are finite critical points to $\Phi$. 
The amplitude is
$${\tilde a}_j(t,r\theta,\lambda)={\tilde a}_{j 0}+{\tilde a}_{j1}+\cdots \in S_{phg}^{0}(U)$$ 
considering $t$ and $\theta$ as $C^\8$ parameter and $U$ is some neighborhood of $r_1$.
We have ${\tilde\Phi }_j=(r-t)+{\cal O}(rt)$ or ${\tilde\Phi }=-(r+t)+{\cal O}(rt)$ from (4.7).  
We also assume ${\pl_r\tilde\Phi_j }\not=0$ by changing $r\sim t$ smaller. If there exists a point $r_2\theta_2$ such that $\tilde\Phi_j(t,r_2\theta_2)\not=0$ then by using the identity
$$\frac{1}{i}\tilde\Phi_j^{-1}\pl_\lambda e^{i\lambda \tilde\Phi_j}=e^{i\lambda \tilde\Phi_j}$$
and integration by parts, we get the integral is  $C^\8$ near this point.
On the other hand, by the propagation of singularity, $A(t,x)$ has only singularity at $r=t,-t$. 
So essential parts of the integral are the phase ${\tilde\Phi }_j$s which are the following form
$${\tilde\Phi }_j^+=(r-t){\hat \Phi _j^+}\text{　or　}{\tilde\Phi }_j^-=-(r+t){\hat \Phi _j}^-.$$
Since ${\pl_r\tilde\Phi_j }\not=0$, we have ${\hat \Phi }^\pm_j\not=0$. Near $r_1$, we get
\begin{align*}
A(t,r\theta)\equiv\sum_j \int_0^\8&e^{i\lambda(r-t){\hat \Phi }^+_j(t,r\theta )} {\tilde a_j}(t,r\theta,\lambda)\lambda^{(d-1)/2}d\lambda\\
&+\sum_j \int_0^\8e^{-i\lambda(r+t){\hat \Phi }^-_j(t,r\theta )} {\tilde a_j}(t,r\theta,\lambda)\lambda^{(d-1)/2}d\lambda
\end{align*}
By changing the variable as $\lambda'=\lambda{\hat \Phi }^{\pm}_j(t,r\theta )$
\begin{align*}
A(t, r\theta)\equiv\int_0^\8&e^{i\lambda'(r-t)} {\hat a}_+(t,r\theta,\lambda'){\lambda'}^{(d-1)/2}d\lambda'\\
&+\int^{\8}_0e^{-i\lambda(r+t)}{\hat a}_-(t,r\theta,\lambda'){\lambda'}^{(d-1)/2}d\lambda'
\end{align*}
where ${\hat a}_{\pm}=\sum_{j}{\tilde a_j}(t,r\theta,\lambda{\hat \Phi }^{\pm}_j){\hat \Phi }_j^{\pm}$. By the propagation of singularity, we know ${\hat a_{\pm}}$ is elliptic. If $t>0$, since $(r+t)\not=0$ the second term is $C^\8$, we can write this integrals as follows
\begin{align*}
A(t,r\theta)&\equiv\int_0^\8e^{ i\lambda'(r- t)} {\hat a}_{+}(t,r\theta,\lambda'){\lambda'}^{(d-1)/2}d\lambda'\\
&=\int_0^\8\int^\8_\8e^{ i\lambda(r- z)} {\hat a}'_{+}(z,r\theta,\lambda'){\lambda'}^{(d-1)/2}\delta(z-t)dzd\lambda'
\end{align*}
If $t<0$, the first term is $C^\8$ and we have
\begin{align*}
A(t,r\theta)&\equiv\int^{\8}_0e^{-i\lambda(r+t) }{\hat a}_-(t,r\theta,\lambda'){\lambda'}^{(d-1)/2}d\lambda'\\
&=\int_{-\8}^0\int^\8_{-\8}e^{i\lambda'(r- z)} {\hat a}'_{-}(z,r\theta,\lambda'){|\lambda'|}^{(d-1)/2}\delta(z-|t|)dzd\lambda'
\end{align*}
Here ${\hat a}'_{-}={\hat a_-}(z,r\theta,|\lambda'|)$. Thus we have
\begin{prop}
Let $M$ be a smooth Riemannian manifold, $x_0 \in M$ and $L$ is a self-adjoint $\Psi$DO satisfies
 $\sigma_{pri}(L)=\sigma_{pri}(\sqrt{-\Delta_M})$. 
We take $x_1\not=x_0$ sufficiently near $x_0$ and the normal coordinates $r\theta$ around $x_0=0$. Then there exists $U$, a neighborhood of $r_1$, 
and ${\hat a}'_{\pm}(z,r\theta,\lambda')\in S_{phg}^{0}(U)$, elliptic symbol with $C^\8$ parameter $\theta$ 
such that for small $t>0$, we have the following identity near $x_1$ modulo smooth functions
\begin{equation}
e^{\pm itL}\delta_{x_0}(r\theta) \equiv
\begin{cases}
\ds \int_0^\8\int^\8_{-\8}e^{ i\lambda'(r- z)} {\hat a}'_{+}(z,r\theta,\lambda'){\lambda'}^{(d-1)/2}\delta(z-t)dzd\lambda'\\
\ds \int_{-\8}^0\int^\8_{-\8}e^{i\lambda'(r- z)} {\hat a}'_{-}(z,r\theta,\lambda'){|\lambda'|}^{(d-1)/2}\delta(z-t)dzd\lambda'
\end{cases}.
\end{equation}
\end{prop}
\begin{remark}
We can get this fact more directory by solving (4.5). 
For short, we use the propagation of singularity.  
\end{remark}
Now we study the singularity of $e^{- itL^2}\delta_{x_0}$ near $x_1$. 
We take a smooth cut-off function $\rho(r\theta)$ near $(r_1, \theta_1)$,
 i.e. $\rho=1$ on a small neighborhood of $(r_1, \theta_1)$ and $\rho$ is identically zero out side another neighborhood of $(r_1, \theta_1)$. 
Then by the identity (2.13) and the propagations of singularity, we have the following identity for $N>1$
\begin{equation}
\rho e^{-itL^2}\delta_{x_0}\equiv \rho \la L\ra^{N}\int^{\pi}_{0}\tilde\rho(y)\{\la D_y\ra^{-N}  G(t,y)\}\{(e^{iyP}+e^{-iyP})\}dy\delta_{x_0}
\end{equation}
where $\tilde\rho(y)\in C^\8([0,\pi]))$ is a cut-off function near $r_0$. By Proposition 4.4, the right hand side can be written as follows in the normal polar coordinate near $x_0$.
\begin{align*}
(e^{iyP}+e^{-iyP})\delta_{x_0}\equiv \frac{1}{2\pi}\int_{-\8}^\8\int_{-\8}^\8e^{i \lambda'(r- z)} b(z,r\theta,\lambda')\delta(z-y)dzd\lambda'\}
\end{align*}
Here we take
\begin{equation}
b(z,r\theta,\lambda')=
\begin{cases}
2\pi{\hat a_+}'(z,r\theta,\lambda'){\lambda'}^{(d-1)/2}\ \ \ \lambda'>0\\
2\pi{\hat a_-}'(z,r\theta,\lambda'){|\lambda'|}^{(d-1)/2}\ \ \    \lambda'<0. 
\end{cases}
\end{equation}
Since $\hat a'_\pm$ is elliptic with respect to $\lambda'$  , $b$ is elliptic. The right hand side of (4.9) becomes
\begin{align*}
\rho & \la L\ra^{N}\int^{\pi}_{0}\tilde\rho(y)\{\la D_y\ra^{-N}  G(t,y)\}  \frac{1}{2\pi}\int_{-\8}^\8\int_{-\8}^\8e^{i \lambda'(r- z)} b(z,r\theta,\lambda')\delta(z-y)dzd\lambda'      dy\\
&= \rho \la L\ra^{N}\frac{1}{2\pi}\int_{-\8}^\8\int_{-\8}^\8e^{i \lambda'(r- z)} b(z,r\theta,\lambda')\int^{\pi}_{0}\tilde\rho(y)\{\la D_y\ra^{-N}  G(t,y)\}\delta(z-y)  dy dzd\lambda'  \\
&= \rho \la L\ra^{N}\frac{1}{2\pi}\int_{-\8}^\8\int_{-\8}^\8e^{i \lambda'(r- z)} b(z,r\theta,\lambda')\{\tilde\rho(y)\la D_y\ra^{-N}  G(t,y)\}|_{y=z}  dzd\lambda' 
\end{align*}
Here the last integral is $\Psi$DO with $C^\8$ parameter $\theta$. We write this operator as $B(\theta)$ and get
\begin{prop}
Let $M$ be a Zoll manifold with period $2\pi$, $x_0 \in M$ and $L$ is a self-adjoint $\Psi$DO defined by (3.5). 
We take $x_1\not=x_0$ sufficiently near $x_0$ and the normal polar coordinates $r\theta$ around $x_0=0$. Then there exists $U$, a neighborhood of $r_1$, 
and $b(z,r\theta,\lambda')\in S^{(d-1)/2}(U)$, elliptic symbol with $C^\8$ parameter $\theta$ 
such that for small $t>0$ and $n>1$, we have the following identity near $x_1$ modulo smooth functions
\begin{equation}
\rho e^{-itL^2}\delta_{x_0}\equiv \la L\ra^{N} B(\theta)( \tilde\rho(y)\la D_y\ra^{-N}  G(t,y))\
\end{equation}
\end{prop}
This proposition reduce the singularity of $e^{-itL^2}\delta_{x_0}$ to $G(t,y)$. 
If $t/2\pi=\alpha\in\R\backslash\Q$, by Proposition 2.9, $\tilde G(y)=\tilde\rho(y)\la D_y\ra^{-N}G(2\pi\alpha,y)$ is singular at $r_1$ to both directions. 
Since $B(t,\theta)$ is elliptic $\Psi$DO with $C^\8$ parameter $\theta$, this preserve the singularity to $r$-variable. We have 
\begin{equation*}
(r_1\theta_1, \pm \theta_1)\in {\rm WF}(e^{-itL^2}\delta_{x_0}).
\end{equation*}
We also know, in normal polar coordinate, 
\begin{equation}
(r_1\theta_1, \pm \theta_1)= \exp (\pm r_1 H_l)(0, \pm\theta_1).
\end{equation}
Since we can take $x_1$ arbitrary near $x_0$, we have
\begin{equation}
(0, \pm \theta_1)\in {\rm WF}(e^{-itL^2}\delta_{x_0}).
\end{equation}
Next we prove the existence of the similar singularity at every points using (1.7). As in the previous argument, we take smooth cut-off functions $\rho(r\theta)$ near $\Pi_x(\exp ((r_1+4\pi\alpha) H_l)(0, \theta_1))$ and $\tilde\rho(y)$ near $r_1+4\pi\alpha$ We have 
\begin{align*}
\rho e^{-i2\pi\alpha L^2}\delta_{x_0}&\equiv \rho \la L\ra^{N}\int^{\pi}_{0}\tilde\rho(y)\{\la D_y\ra^{-N}  G(2\pi\alpha,y)\}(e^{iyL}+e^{-iyL})dy\delta_{x_0}\\
&=\rho\la L\ra^{N}\int^{\pi}_0 \tilde\rho(y)\{\la D_y\ra^{-N}  e^{i(2\pi\alpha+y)}G(2\pi\alpha,y+4\pi\alpha)\}  (e^{iyL}+e^{-iyL})\delta_{x_0}dy.
\end{align*}
Change the variable $y=y'-4\pi\alpha$. 
\begin{align*}
&=\rho\la L\ra^{N}\int^{\pi}_0 \tilde\rho(y'-4\pi\alpha)\{\la D_y\ra^{-N}  e^{i(y'-2\pi\alpha)}G(2\pi\alpha,y')\}  (e^{i(y'-4\pi\alpha)L }
+e^{-i(y'-4\pi\alpha)L})\delta_{x_0}dy'\\
&=\rho e^{i4\pi\alpha L}\la L\ra^{N}\int^{\pi}_0 \{\tilde\rho(y'-4\pi\alpha)\la D_y\ra^{-N}  e^{i(y'-2\pi\alpha)}G(2\pi\alpha,y')\} e^{i  y'L }\delta_{x_0}dy'\\
&\ \ +\rho e^{-i4\pi\alpha L}\la L\ra^{N}\int^{\pi}_0 \{\tilde\rho(y'-4\pi\alpha)\la D_y\ra^{-N}  e^{i(y'-2\pi\alpha)}G(2\pi\alpha,y')\} e^{- iy'L }\delta_{x_0}dy'.
\end{align*}
We write this as follows
\begin{equation}
\rho e^{-i2\pi\alpha L^2}\delta_{x_0}\equiv \rho e^{i4\pi\alpha L}{\rm I}+ \rho e^{-i4\pi\alpha L} \ {\rm II}.
\end{equation}
Now $\hat G(y')=\tilde\rho(y'-4\pi\alpha)\la D_y\ra^{-N}  e^{i(y'-2\pi\alpha)}G(2\pi\alpha,y')$ is singular to both directions at $y'=r_1$ and applying previous argument. We write (I) as $B_1(\theta)\hat G(y')$.
Since $B_1(\theta)$ a $\Psi DO$ elliptic to the positive direction $r>0$, we have   
\begin{align*}
(r_1\theta_1, \theta_1)\in {\rm WF}(\rm I)
\end{align*}
Taking $r_1$ sufficient small, we also assume
$$-r_1-8\pi\alpha\not\equiv r_1+4\pi\alpha , \ \ {\rm mod }\ 2\pi $$
By shrinking the support of $\rho$, we have
\begin{align*}
\exp (-8\pi\alpha H_l)(r_1\theta_1, -\theta_1)\not \in {\rm WF}(\rm I).
\end{align*}
By the propagation of the singularity, we get
\begin{equation}
\begin{split}
\exp (4\pi\alpha H_l)(r_1\theta_1, \theta_1)\in {\rm WF}(\rho e^{i4\pi\alpha L}\rm I)\\
\exp (-4\pi\alpha H_l)(r_1\theta_1, -\theta_1)\not \in {\rm WF}(\rho e^{i4\pi\alpha L}\rm I).
\end{split}
\end{equation}
Similarly we obtain
\begin{equation}
\begin{split}
\exp (8\pi\alpha H_l)(r_1\theta_1, \theta_1)\not\in {\rm WF}(\rho e^{-i4\pi\alpha L}\rm II)\\
\exp (-4\pi\alpha H_l)(r_1\theta_1, -\theta_1)\in {\rm WF}(\rho e^{-i4\pi\alpha L}\rm II).
\end{split}
\end{equation}
(4.12), (4.13), (4.14) and (4.15) mean
\begin{align*}
\exp ((r_1\pm4\pi\alpha) H_l)(0, \pm\theta_1)\in {\rm WF}(e^{-i2\pi\alpha L^2}\delta_{x_0}).
\end{align*}
Since $r_1>0$ is any small constant we have
\begin{align*}
\exp ((\pm4\pi\alpha) H_l)(0, \pm\theta_1)\in {\rm WF}(e^{-i2\pi\alpha L^2}\delta_{x_0}).
\end{align*}
Inductively we can prove similar singularity to the points $\exp ((4\pi\alpha k) H_l)(0, \pm\theta_1)$, $\forall k\in\Z$, $\forall\theta_1$.  Since $4\pi\alpha k$, $\in\Z$ are dense in $S^1$, we get
\begin{equation}
\cup_{0\leq \tau \leq 2\pi} {\rm exp}(\tau H_p) {\rm{WF}}(\delta_{x_0})\subset {\rm{WF}}(e^{-i2\pi\alpha L^2}\delta_{x_0})
\end{equation}
Here we use periodicity, $\exp (-t H_l)= \exp ((2\pi -t )H_l)$. 
By Proposition 3.5 and (4.2), this implies (4.1).
\begin{small}
\begin{center}

\end{center}

Department of Mathematics, Graduate School of Science, Osaka University,

1-1 Machikaneyama-cho, Toyonaka, Osaka 560-0043, Japan 

E-mail:h-nishiyama@cr.math.sci.osaka-u.ac.jp
\end{small}
\end{document}